\documentclass[a4papersize]{article}

\usepackage{amsmath,amsthm, amsxtra,amssymb,latexsym, amscd,enumerate}
\usepackage[mathscr]{eucal}
\newtheorem{theorem}{Theorem}[section]

\newtheorem{lemma}[theorem]{Lemma}

\DeclareMathOperator{\Ker}{Ker}
\DeclareMathOperator{\Hom}{Hom}
\DeclareMathOperator{\Ann}{Ann}

\DeclareMathOperator{\Ass}{Ass}

\DeclareMathOperator{\Var}{Var}

\DeclareMathOperator{\Rad}{Rad}
\DeclareMathOperator{\Att}{Att}

\DeclareMathOperator{\Spec}{Spec}
\DeclareMathOperator{\p}{\frak p}
\DeclareMathOperator{\q}{\frak q}
\DeclareMathOperator{\m}{\frak m}
\DeclareMathOperator{\R}{\widehat R}

\begin{document}
\large
\centerline{\Large {\bf ATTACHED PRIMES OF  LOCAL COHOMOLOGY MODULES }}
\medskip
\centerline{\Large {\bf UNDER LOCALIZATION AND COMPLETION }}

\vskip 0.7cm

\centerline { LE THANH NHAN }
\centerline{Thai Nguyen College of Sciences}
\centerline{Thai Nguyen University, Thai Nguyen, Vietnam}
\centerline{E-mail: trtrnhan@yahoo.com}
\medskip
\centerline {PHAM HUNG QUY}
\centerline{Department of Mathematics,  FPT University}
\centerline{8 Ton That Thuyet Road, Hanoi, Vietnam}
\centerline{E-mail:quyph@fpt.edu.vn}

\vskip 1cm

\noindent{\bf Abstract} {\footnote{ {\it{Key words and phrases: }} Universally catenary rings whose formal fibers are Cohen-Macaulay, attached primes of  local cohomology modules. \hfill\break
  {\it{2010 Subject  Classification: }} 13D45, 13H10, 13E05.\hfill\break {The authors are supported by the Vietnam National Foundation for Science and Technology Development (Nafosted). This paper was written while the authors visited Vietnam  Institute for advanced study in Mathematics (VIASM), they would like to thank VIASM  for the very kind support and hospitality.}}. Let $(R,\m)$ be a Noetherian local ring and $M$ a finitely generated $R$-module. Following I. G. Macdonald \cite{Mac},  the set of all attached primes of the Artinian local cohomology module  $H^i_{\m}(M)$ is denoted by $\Att_R(H^i_{\m}(M))$.  In \cite[Theorem 3.7]{Sh}, R. Y. Sharp proved that if $R$ is a quotient of a Gorenstein local ring then the shifted localization principle always holds true, i.e.   
$$ \ \ \ \ \ \ \ \ \ \ \ \ \ \ \ \ \ \ \ \ \ \ \ \ \Att_{R_{\p}}\big(H^{i-\dim (R/\p)}_{\p R_{\p}}(M_{\p})\big)=\big\{\q R_{\p}\mid \q\in\Att_RH^i_{\m}(M), \q\subseteq \p\big\}  \ \ \ \ \ \ \ \ \ \ \ \ \ \ \ \ \ \ \ \ \   (1)$$
for any local cohomology modules  $H^i_{\m}(M)$ and any $\p\in\Spec (R).$ In this paper, we improve Sharp's result as follows:  the shifted localization principle always holds true if and only if $R$ is universally catenary and all its formal fibers  are  Cohen-Macaulay, if and only if   
$$\ \ \ \ \ \ \ \ \ \ \ \ \ \ \ \ \ \ \ \ \ \ \ \ \ \ \ \ \ \ \ \ \ \ \ \ \ \ \  \displaystyle \Att_{\R}(H^i_{\m}(M))=\bigcup_{\p\in\Att_R(H^i_{\m}(M))}\Ass_{\R}(\R/\p\R)\ \ \ \ \ \ \ \ \ \ \ \ \ \ \ \ \ \ \ \ \ \ \ \ \ \ \ \ \ \ \ \ \ \ \ \ \ (2)$$
holds true for  any finitely generated $R$-module $M$ and any integer $i\geq 0.$ This also improves the main result of the paper \cite{CN}.  
 
 \section{Introduction}

 Throughout this paper, let $(R,\m )$ be a Noetherian local ring and $M$ a finitely generated $R$-module with $\dim M=d$. It is well kwown that $$\Ass_{R_{\p}}(M_{\p})=\big\{\q R_{\p}\mid  \q\in\Ass_RM, \q\subseteq \p\big\}$$ for every prime ideal $\p$ of $R$.  For an Artinian $R$-module $A$, the set of all attached primes $\Att_RA$  defined by I. G. Macdonald \cite{Mac}  makes  an important role similarly to the role of the set of associated primes $\Ass_RM$ of a finitely generated $R$-module $M$. It is well known that the local cohomology module $H^i_{\m}(M)$ is Artinian for all $i\geq 0$. Therefore, it is natural to ask whether the analogous relation 
$$ \ \ \ \ \ \ \ \ \ \ \ \ \ \ \ \ \ \ \ \ \ \ \ \ \ \ \ \  \Att_{R_{\p}}\big(H^{i-\dim (R/\p)}_{\p R_{\p}}(M_{\p})\big)=\big\{\q R_{\p}\mid \q\in\Att_R(H^i_{\m}(M)), \q\subseteq \p\big\}  \ \ \ \ \ \ \ \ \ \ \ \ \ \ \ \ \ \ \ \   (1)$$
between $\Att_{R_{\p}}\big(H^{i-\dim (R/\p)}_{\p R_{\p}}(M_{\p})\big)$ and $\Att_R(H^i_{\m}(M))$ holds true for every integer $i$ and every  $\p\in\Spec (R).$  If $R$ is a quotient of a Gorenstein local ring,  R. Y. Sharp \cite[Theorem 3.7]{Sh} proved that  (1) always holds true  (see also \cite[11.3.2]{BS}).    However, this relation does not hold true in general, cf. \cite[Example 11.3.14]{BS}.

Another question is about the relation between the attached primes of $H^i_{\m}(M)$ over $R$  and that of $H^i_{\m}(M)$ over the $\m$-adic completion $\R$ of $R$. Denote by $\widehat M$  the $\m$-adic completion of  $M$. Then we have following well known  relations between $\Ass_RM$ and $\Ass_{\R}\widehat M$
$$ \Ass_RM =\big\{\frak P\cap R\mid\frak P\in\Ass_{\R}\widehat M\big\} \ \text{and}\  \Ass_{\R}\widehat M=\displaystyle\bigcup_{\p\in\Ass M}\Ass_{\R}(\R/\p\R),$$
 cf. \cite[Theorem 23.2]{Mat}. For an Artinian $R$-module $A$, we note that $A$ has a natural structure as an Artinian  $\R$-module. Moreover, $\Att_RA=\{\frak P\cap R\mid \frak P\in\Att_{\R}A\}$ (see \cite[8.2.4, 8.2.5]{BS}), which is in some sense dual to  the  above first relation between $\Ass_RM$ and $\Ass_{\R}\widehat M.$ However,  the second analogous relation may not hold true even when $A=H^i_{\m}(M),$ i.e. the following relation
 $$\ \ \ \ \ \ \ \ \ \ \ \ \ \ \ \ \ \ \ \ \ \ \ \ \ \ \ \ \ \ \ \ \ \ \ \ \ \ \  \displaystyle \Att_{\R}(H^i_{\m}(M))=\bigcup_{\p\in\Att_R(H^i_{\m}(M))}\Ass_{\R}(\R/\p\R).\ \ \ \ \ \ \ \ \ \ \ \ \ \ \ \ \ \ \ \ \ \ \ \ \ \ \ \ \ \ \ \ \ \ \ \ \ (2)$$
is not true in general, cf. \cite[Example 2.3]{CN}.

 In this paper,  we study  attached primes of $H^i_{\m}(M)$ under  localization and $\m$-adic completion. We prove that (1) and (2) are both equivalent to  the condition that the base ring $R$ is universally catenary and all formal fibers of $R$ are Cohen-Macaulay.  The following theorem is the main result of this paper.

\begin{theorem} \label{T:1} The following statements are equivalent:
\begin{enumerate}[{(i)}]\rm
\item {\it $R$ is universally catenary and all its formal fibers are Cohen-Macaulay;}
\item {\it $\displaystyle  \Att_{R_{\p}}\big(H^{i-\dim (R/\p)}_{\p R_{\p}}(M_{\p})\big)=\big\{\q R_{\p}\mid \q\in\Att_R(H^i_{\m}(M)), \q\subseteq \p\big\}$ for every finitely generated $R$-module $M$, integer $i\geq 0$ and prime ideal $\p$ of $R$;}
\item {\it $\displaystyle\Att_{\R}(H^i_{\m}(M))=\bigcup_{\p\in\Att_R(H^i_{\m}(M))}\Ass_{\R}(\R/\p\R)$ for every finitely generated $R$-module $M$ and  integer $i\geq 0$.}
  \end{enumerate}
\end{theorem}

It should be mentioned that the condition (i) in Theorem \ref{T:1} is equivalent to the condition that $R$ is a quotient of a Cohen-Macaulay local ring, cf. \cite[Corollary 1.2]{Kaw}. Recently,  N. T. Cuong and D. T. Cuong \cite{CC} proved that the condition (i) in Theorem \ref{T:1} is  equivalent to the existence of a p-standard system of parameters of $R$.

In the next section, we give some preliminaries on attached primes of Artinian modules that will be used in the sequel. We prove the main result of this paper (Theorem \ref{T:1})  in the last section.

\section{Preliminaries}

I. G. Macdonald \cite{Mac} introduced the theory of secondary representation for Artinian modules, which is in some sense dual to the theory of primary decomposition for Noetherian modules. Let $A\neq 0$ be an Artinian $R$-module. We say that $A$ is {\it secondary} if the multiplication by $x$ on $A$ is surjective or nilpotent for every $x\in R.$ In this case, the set $\p :=\Rad (\Ann_RA)$ is a prime ideal of $R$ and we say that $A$ is {\it $\p$-secondary}. Note that every Artinian $R$-module $A$ has a minimal secondary representation $A=A_1+\ldots+A_n,$ where $A_i$ is $\p_i$-secondary, each $A_i$ is not redundant and $\p_i\neq \p_j$ for all $i\neq j.$ The set $\{\p_1,\ldots,\p_n\}$ is independent of the choice of the minimal secondary representation of $A$. This set is called the set of {\it attached primes} of $A$ and denoted by $\Att_RA$. 

For each ideal $I$ of $R$, we denote by $\Var(I)$  the set of all prime ideals of $R$ containing $I$. 

\begin{lemma}\label {L:1a}{\rm (\cite{Mac})}. The following statements are true.
\begin{enumerate}[{(i)}]\rm
\item {\it $A\neq 0$ if and only if $\Att_RA\neq \emptyset.$}
\item {\it $\min \Att_RA=\min \Var(\Ann_RA).$ In particular,  $$\dim (R/\Ann_RA)=\max\{\dim(R/\p)\mid\p\in\Att_RA\}.$$}
\item {\it If $0\rightarrow A'\rightarrow A\rightarrow A''\rightarrow 0$ is an exact sequence of Artinian $R$-modules then $$\Att_RA''\subseteq \Att_RA\subseteq \Att_RA'\cup\Att_RA''.$$}
 \end{enumerate}
\end{lemma}

Note that  $A$ has a natural structure as an $\R$-module and with this structure, each subset of $A$ is an $R$-submodule if and only if it is an $\R$-submodule. Therefore $A$ is an Artinian $\R$-module. So, the set of attached primes $\Att_{\R}A$ of $A$ over $\R$ is well defined.

\begin{lemma} \label {L:2a} {\rm(\cite[8.2.4, 8.2.5]{BS}).}  $\Att_RA=\big\{\mathfrak{P}\cap R \mid \mathfrak{P}\in\Att_{\R}A\big\}.$
\end{lemma}

\begin{lemma}   \label{L:3a} Let $A$ be an Artinian $R$-module. Let $(S,\frak n)$ be a Noetherian local ring and let  $\varphi: R\rightarrow S$ be a flat local homomorphism between local rings $(R,\m)$ and $(S,\frak n).$ Suppose that  $\dim (S/\m S)=0.$ Then $A\otimes_R S$ is an Artinian $S$-module and
$$\Att_RA=\{\varphi^{-1}(\frak P)\mid \frak P\in\Att_S(A\otimes S)\}.$$
\end{lemma}
\begin{proof} Firstly we use Melkersson's criterion \cite[Theorem 1.3]{Mel} to prove $A\otimes_RS$ is an Artinian $S$-module.  Since $S$ is flat over $R$ and $R/\m$ is of finite representation, we get by \cite[Theorem 7.11]{Mat} that
$$\Hom_S(S/\m S;A\otimes_RS)\cong \Hom_S(R/\m\otimes_RS ; A\otimes_RS)\cong \Hom_R(R/\m; A)\otimes_RS.$$
Because $A$ is an Artinian $R$-module, $\Hom_R(R/\m;A)$ is an $R$-module of finite length. Hence $\Hom_R(R/\m;A)$ is a finitely generated $R$-module. Therefore $ \Hom_R(R/\m; A)\otimes_RS$ is a finitely generated $S$-module  which is annihilated by $\m S.$   Because  $\dim (S/\m S)=0,$ it follows that $\Hom_R(R/\m; A)\otimes_RS$ is an $S$-module of finite length.
Since $A$ is $\m$-torsion, it is obvious to see that $A\otimes_R S$ is $\m S$-torsion. Therefore $A\otimes_R S$ is an Artinian $S$-module.

Let $A=A_1+\ldots +A_n$ be a minimal secondary representation of $A$, where $A_i$ is  $\p_i$-secondary for $i=1,\ldots , n.$  Then $\Att_RA=\{\p_1,\ldots , \p_n\}.$ As $S$ is a faithfully flat $R$-algebra, $R$ can be considered as a subring of $S$ and  $A_i\otimes_R S$ can be considered as a submodule of $A\otimes_RS$ for all $i=1,\ldots ,n.$  Then we have $A\otimes_R S=(A_1\otimes_RS)+\ldots +(A_n\otimes S).$ For each $i=1,\ldots ,n,$ choose a minimal secondary representation $A_i\otimes_RS=B_{i1}+\ldots +B_{ik_i}$  of $S$-module $A_i\otimes_RS,$  where $B_{ij}$ is $\frak P_{ij}$-secondary.  Then  $\displaystyle A\otimes_RS=\sum_{i=1}^n(B_{i1}+\ldots +B_{ik_i})$ is a secondary representation of $A\otimes_RS$. By removing all redundant components and then remumbering the components, we can assume that there exists an integer $t_i\leqslant k_i$ for $i=1,\ldots ,n$ such that  $\displaystyle A\otimes_RS=\sum_{i=1}^n(B_{i1}+\ldots +B_{it_i})$ is a secondary representation of $\displaystyle A\otimes_RS$ without any redundant component.   Since $A_i$ is not redundant in the secondary representation  $A=A_1+\ldots +A_t$ and $S$ is faithfully flat over $R$, we have $t_i\geq 1$ for all $i=1,\ldots ,n.$  Now let  $i\in\{1,\ldots ,n\}$ and let $x\in\p_i$. Then $x^mA_i=0$ for some $m\in \Bbb N$. Hence $x^m(A_i\otimes_RS)=0$ and hence $x^mB_{ij}=0$ for all $j=1,\ldots ,t_i.$ Therefore $x\in\frak P_{ij}\cap R$ for all $j=1,\ldots ,t_i.$ Let $x\in R\setminus \p_i.$ Then $x^mA_i=A_i$ and hence $x^m(A_i\otimes_RS)= A_i\otimes_RS$ for all $m\in\Bbb N.$ If $x\in\frak P_{ij}$ for some $j\in\{1,\ldots ,t_i\}$ then  $x^{m_0}B_{ij}=0$ for some $m_0\in\Bbb N$ and hence $x^{m_0}(A_i\otimes_RS)\neq A_i\otimes_RS$, this is a contradiction. Therefore $x\notin \frak P_{ij}$ for all $j=1,\ldots ,t_i.$ It follows that $\p_i=\frak P_{ij}\cap R$ for all $j=1,\ldots ,t_i.$ Hence $\frak P_{ij}$'s are pairwise different and hence   $\displaystyle A\otimes_RS=\sum_{i=1}^n(B_{i1}+\ldots +B_{it_i})$ is a minimal secondary representation of $\displaystyle A\otimes_RS$. Therefore $\Att_S(A\otimes_RS)=\{\frak P_{ij}\mid i=1,\ldots ,n, j=1,\ldots ,t_i\}.$ Thus $$\Att_RA=\{\frak P\cap R\mid \frak P\in\Att_S(A\otimes_R S)\}.$$
\end{proof}

\section{Main results}

 An important step to prove the main result of this paper is to find for each integer $i<d$ and each attached prime $\p\in \Att_{R}(H^i_{\m}(M))$ a suitable finitely generated $R$-module $N$ such that $\p\in\Ass_{R}N$ (see Lemma \ref{L:3}).  This step can be done by using a splitting property for local cohomology modules proved by N. T. Cuong and P. H. Quy \cite[Corollary 3.5]{CQ1} (see Lemma \ref{L:1}).  It should be mentioned that this  splitting property is an extension of the original splitting result  \cite[Theorem 1.1]{CQ}.

From now on, for a subset $T$ of $\Spec (R)$ and an integer $i\geq 0$, we  set 
$$(T)_i=\{\p\in T\mid \dim (R/\p)=i\}.$$   For a finitely generated $R$-module $N$ of dimension $t>0$,  we set $\frak a_i(N)=\Ann_R(H^i_{\m}(N))$ for $i=0,\ldots ,t$ and   $\frak a(N)=\frak a_0(N)\ldots \frak a_{t-1}(N).$  Note that 
 $$\displaystyle \frak a(N)\subseteq \underset{\underline x}{\bigcap}\ \bigcap_{i=1}^t\Ann_R(0:_{N/(x_1,\ldots ,x_{i-1})N}x_i),$$ where $\underline x=(x_1,\ldots ,x_t)$ runs over the set of all systems of parameters of $N$, cf. \cite[Satz 2.4.5]{Sch}. Therefore,  by \cite[Corollary 3.5]{CQ1} we have the following splitting result.

\begin{lemma}  \label{L:1} Set $\overline M=M/U_M(0)$, where $U_M(0)$  is the largest submodule of $M$ of dimension less than $d$. Suppose that  $x\in\frak a(M)^3$  is a parameter element of $M$. Then for all $i<d-1$  we have 
$$H^i_{\m}(M/xM)\cong H^i_{\m}(M)\oplus H^{i+1}_{\m}(\overline{M}).$$
\end{lemma}

By using Lemma \ref{L:1}, we have the following property, which is needed in the induction step of the proof of Lemma \ref{L:3}.

\begin{lemma} \label{L:2} Suppose that $x\in\frak a(M)^3$ is a parameter element of $M$.  Then we have
$$\bigcup_{i=0}^{d-1}\Att_R(H^i_{\m}(M))\subseteq \bigcup_{i=0}^{d-2}\Att_R(H^i_{\m}(M/xM))\cup (\Ass_RM)_{d-1}.$$\end{lemma}

\begin{proof}  Denote by $U_M(0)$  the largest submodule of $M$ of dimension less than $d$. Set $\overline M=M/U_M(0)$.  Firstly we claim that $$\Att_R(H^{d-1}_{\m}(M))=(\Ass_RM)_{d-1}\cup \Att_R(H^{d-1}_{\m}(\overline M)).$$ In fact, from the exact sequence $0\rightarrow U_M(0)\rightarrow M\rightarrow \overline M\rightarrow 0$ we have the exact sequence $$H^{d-1}_{\m}(U_M(0))\overset f{\rightarrow} H^{d-1}_{\m}(M)\rightarrow H^{d-1}_{\m}(\overline M)\rightarrow 0.$$ If $\dim (U_M(0))<d-1$ then $\Att_R(H^{d-1}_{\m}(U_M(0)))=\emptyset=(\Ass_RM)_{d-1}$ by Lemma \ref{L:1a}(i). Otherwise, we have $\dim (U_M(0))=d-1$, and hence $$\Att_R(H^{d-1}_{\m}(U_M(0)))=(\Ass_RU_M(0))_{d-1}=(\Ass_RM)_{d-1}$$ by \cite[7.3.2]{BS}. Therefore, it follows by Lemma \ref{L:1a}(iii) that
\begin{align}\Att_R(H^{d-1}_{\m}(M))&\subseteq \Att_R\big(H^{d-1}_{\m}(U_M(0))/\Ker f\big)\cup \Att_R(H^{d-1}_{\m}(\overline M))\notag\\
&\subseteq \Att_R(H^{d-1}_{\m}(U_M(0)))\cup \Att_R(H^{d-1}_{\m}(\overline M))\notag\\
&=(\Ass_RM)_{d-1}\cup \Att_R(H^{d-1}_{\m}(\overline M)).\notag\end{align}
Since  $(\Ass_RM)_{d-1}\subseteq \Att_R(H^{d-1}_{\m}(M))$ by \cite[11.3.9]{BS} and $\Att_R(H^{d-1}_{\m}(\overline M))\subseteq \Att_R(H^{d-1}_{\m}(M))$ by the above exact sequence, it follows that 
$$\Att_R(H^{d-1}_{\m}(M))=(\Ass_RM)_{d-1}\cup \Att_R(H^{d-1}_{\m}(\overline M)).$$ So, the claim is proved.

Now, it follows by Lemma \ref{L:1} that 
$$\bigcup_{i=0}^{d-2}\Att_R(H^i_{\m}(M/xM))=\bigcup_{i=0}^{d-2}\Big(\Att_R(H^i_{\m}(M))\cup \Att_R(H^{i+1}_{\m}(\overline M))\Big).$$  Note that $H^0_{\m}(\overline M)=0$. Therefore we get by the above claim that 
\begin{align}\bigcup_{i=0}^{d-2}&\Att_R(H^i_{\m}(M/xM))\cup (\Ass_RM)_{d-1}=\notag\\
&=\bigcup_{i=0}^{d-2}\Big(\Att_R(H^i_{\m}(M))\cup \Att_R(H^i_{\m}(\overline M))\Big)\cup \Big((\Ass_RM)_{d-1}\cup \Att_R(H^{d-1}_{\m}(\overline M))\Big)\notag\\
&=\bigcup_{i=0}^{d-1}\Big(\Att_R(H^i_{\m}(M))\cup \Att_R(H^i_{\m}(\overline M))\Big).\notag\end{align}
From this, it is obvious to see that
$$\bigcup_{i=0}^{d-1}\Att_R(H^i_{\m}(M))\subseteq \bigcup_{i=0}^{d-2}\Att_R(H^i_{\m}(M/xM))\cup (\Ass_RM)_{d-1}.$$
\end{proof}

The following lemma can be considered as the key lemma for the proof of the main result of this paper.

\begin{lemma} \label {L:3} Let $(x_1,\ldots ,x_d)$ be a system of parameters of $M$ such that for all $i=1,\ldots ,d$ we have  $x_i\in\frak a(M/(x_1,\ldots ,x_{i-1})M)^3.$  Then we have
$$\bigcup_{i=0}^{d-1}\Att_R(H^i_{\m}(M))\subseteq \bigcup_{i=0}^{d-1}\Big(\Ass_R(M/(x_1,\ldots ,x_i)M)\Big)_{d-i-1}.$$
\end{lemma}

\begin{proof} We prove the lemma by induction on $d.$ Let $d=1$. Then the left hand side is $\Att_R(H^0_{\m}(M))$ and the right hand side is $(\Ass_RM)_0.$ So the result is clear. Let $d>1.$  Set $M_1=M/x_1M.$ Then we have by Lemma \ref{L:2} and by induction that
\begin{align}\bigcup_{i=0}^{d-1}\Att_R(H^i_{\m}(M))&\subseteq \bigcup_{i=0}^{d-2}\Att_R(H^i_{\m}(M_1))\cup (\Ass_RM)_{d-1}\notag\\
&\subseteq \bigcup_{i=0}^{d-2}\Big(\Ass_R(M_1/(x_2,\ldots ,x_{i+1})M_1)\Big)_{d-i-2}\cup (\Ass_RM)_{d-1}\notag\\
&= \bigcup_{i=1}^{d-1}\Big(\Ass_R(M/(x_1,\ldots ,x_i)M)\Big)_{d-i-1}\cup (\Ass_RM)_{d-1}\notag\\
&=\bigcup_{i=0}^{d-1}\Big(\Ass_R(M/(x_1,\ldots ,x_i)M)\Big)_{d-i-1}.\notag\end{align}
\end{proof}

 It is  known that $\Ass_R(M_{\p})=\{\q R_{\p}\mid \frak q\in\Ass_RM, \q\subseteq \p\}$ for every prime ideal $\p$ of $R$. However, such an analogous relation between the sets $\Att_{R_{\p}}(H^{i-\dim (R/\p)}_{\p R_{\p}}(M_{\p}))$ and  $\Att_R(H^i_{\m}(M))$  is not true in general, cf. \cite[Example 11.3.14]{BS}. We have the following result which is called {\it the shifted localization principle} for local cohomology modules, cf. \cite[11.3.2]{BS}, \cite[Theorem 3.7]{Sh}.

 \begin{lemma} \label {L:4}   Suppose that $R$ is a quotient of a Gorenstein local ring. Then for any prime ideal $\p$ of $R$ and any  integer $i\geq 0$ we have
$$\Att_{R_{\p}}\big(H^{i-\dim (R/\p)}_{\p R_{\p}}(M_{\p})\big)=\{\q R_{\p}\mid \q\in\Att_R(H^i_{\m}(M)), \q\subseteq \p\}.$$ 
\end{lemma}

In general, we have
$$\Att_{R_{\p}}\big(H^{i-\dim (R/\p)}_{\p R_{\p}}(M_{\p})\big)\subseteq \{\q R_{\p}\mid \q\in\Att_R(H^i_{\m}(M)), \q\subseteq \p\}$$ 
 for any prime ideal $\p$ of $R$ and any integer $i\geq 0$. This later inclusion is called  the {\it weak general shifted localization princile}, cf. \cite[11.3.8]{BS}.

Now we present the  main result of this paper.

\begin{theorem} \label{T:1a} The following statements are equivalent:
\begin{enumerate}[{(i)}]\rm
\item {\it $R$ is universally catenary and all its formal fibers are Cohen-Macaulay;}
\item {\it $\displaystyle  \Att_{R_{\p}}\big(H^{i-\dim (R/\p)}_{\p R_{\p}}(M_{\p})\big)=\big\{\q R_{\p}\mid \q\in\Att_R(H^i_{\m}(M)), \q\subseteq \p\big\}$ for every finitely generated $R$-module $M$, integer $i\geq 0$ and prime ideal $\p$ of $R$;}
\item {\it $\displaystyle\Att_{\R}(H^i_{\m}(M))=\bigcup_{\p\in\Att_R(H^i_{\m}(M))}\Ass_{\R}(\R/\p\R)$ for every finitely generated $R$-module $M$ and  integer $i\geq 0$.}
  \end{enumerate}
\end{theorem}

\begin{proof}  Let $i\geq 0$ be an integer. Firstly we claim that if  there exists a system of parameters $(x_1,\ldots , x_d)$ of $M$ such that $x_k\in\frak a(M/(x_1,\ldots ,x_{k-1})M)^3$ for all $k=1,\ldots , d$ then 
$$\displaystyle \Att_{\R}(H^i_{\m}(M))\subseteq \bigcup_{\p\in\Att_R(H^i_{\m}(M))}\Ass_{\R}(\R/\p\R).$$ 
In fact,  let $\frak P\in \Att_{\R}(H^i_{\m}(M)).$  If $i=d$ then $\frak P\in (\Ass_{\R}\widehat M)_d$ by \cite[Theorem 7.3.2]{BS}. Therefore we get by  \cite[Theorem 23.2]{Mat} that
$$\frak P\in \bigcup_{\p\in (\Ass_RM)_d}\Ass (\R/\p\R)=\bigcup_{\p\in \Att_R(H^d_{\m}(M))}\Ass (\R/\p\R).$$ So, the result is true in this case. Suppose that $i<d.$  It is clear that $$\frak a(M/(x_1,\ldots ,x_{k-1})M)\R\subseteq \frak a(\widehat M/(x_1,\ldots ,x_{k-1})\widehat M)$$ for all $k=1,\ldots ,d.$ Set $\dim (\R/\frak P)=t$. Then $\frak P\in \big(\Att_{\R}(H^i_{\m}(M))\big)_t$. Since $i<d$,   we get  by Lemma \ref{L:3} that $\frak P\in \big(\Ass_{\R}(\widehat M/(x_1,\ldots ,x_{d-t-1})\widehat M)\big)_t$.  Set $\p_0=\frak P\cap R.$ Then we have $\p_0\in\Att_R(H^i_{\m}(M))$  by  Lemma \ref{L:2a}  and $\p_0\in\Ass_R(M/(x_1,\ldots ,x_{d-t-1})M)$. Therefore, it follows by \cite[Theorem 23.2]{Mat} that 
$$\frak P\in \Ass_{\R}(\widehat M/(x_1,\ldots ,x_{d-t-1})\widehat M)=\underset{\p\in\Ass_R(M/(x_1,\ldots ,x_{d-t-1})M)}{\bigcup}\Ass (\R/\p\R).$$  Hence $\frak P\in\Ass (\R/\p_0\R).$ Thus, the claim is proved.

Now we prove  (i) $\Rightarrow$ (ii).  Let $i\geq 0$ be an integer and let $\p$ be a prime ideal of $R$. By the weak general localization principle \cite[11.3.8]{BS}, it is enough to show that if $\q\in\Att_R(H^i_{\m}(M))$ such that $\q\subseteq \p$ then $\q R_{\p}\in \Att_{R_{\p}}\big(H^{i-\dim (R/\p)}_{\p R_{\p}}(M_{\p})\big)$. In fact,  there exists by Lemma \ref{L:2a}  a prime ideal $\frak Q\in \Att_{\R}(H^i_{\m}(M))$ such that $\frak Q\cap R=\q.$  Since $R$ is universally catenary and all its formal fibers are Cohen-Macaulay, we have $\dim (R/\frak a_j(M))\leqslant i$ for all $j=0,\ldots ,d-1$, cf. \cite[Corollary 4.2(i)]{CNN}. Hence $\dim (R/\frak a(M))<d.$  Therefore there exists  an element $x_1\in\frak a(M)^3$ which is a parameter element of $M$. By similar reasons,  we can choose a system of parameters $(x_1,\ldots , x_d)$ of $M$ such that $x_k\in\frak a(M/(x_1,\ldots ,x_{k-1})M)^3$ for all $k=1,\ldots , d.$ So, we get by the above claim that $\displaystyle \frak Q\in\Ass_{\R}(\R/\q\R).$ Note that $R/\q$ is unmixed by  the hypothesis (i), therefore $\dim (\R/\frak Q)=\dim (R/\q).$ Since $\frak Q\in \Att_{\R}H^i_{\m}(M)=\Att_{\R}H^i_{\m\R}(\widehat M)$, we get by Lemma \ref{L:4} that $\frak Q \R_{\frak Q}\in \Att_{\R_{\frak Q}}H^{i-\dim (\R/\frak Q)}_{\frak Q\R_{\frak Q}}(\widehat M_{\frak Q})$. Note that the natural map $R_{\q}\rightarrow \R_{\frak Q}$ is faithfully flat and  $\dim (\R_{\frak Q}/\q \R_{\frak Q})=0$. Moreover, we get by Flat Base Change Theorem \cite[4.3.2]{BS} that
$$ H^{i-\dim (R/\q)}_{\q R_{\q}}(M_{\q})\otimes \R_{\frak Q}\cong H^{i-\dim (\R/\frak Q)}_{\frak Q\R_{\frak Q}}(\widehat M_{\frak Q}).$$  Therefore $\q R_{\q}\in \Att_{R_{\q}}\big(H^{i-\dim (R/\q)}_{\q R_{\q}}(M_{\q})\big)$  by Lemma \ref{L:3a}.  Because $R$ is catenary by the assumption (i), we have $$i-\dim (R/\q)=(i-\dim (R/\p))-\dim (R_{\p}/\q R_{\p}).$$ Therefore, from the fact that $(R_{\p})_{\q R_{\p}}\cong R_{\q},$   we get $\q R_{\p}\in \Att_{R_{\p}}\big(H^{i-\dim (R/\p)}_{\p R_{\p}}(M_{\p})\big)$ by  the weak general shifted localization principle \cite[11.3.8]{BS}.

 (ii) $\Rightarrow$ (iii). Let $\p\in\Att_R(H^i_{\m}(M))$ and $\frak P\in \Ass (\R/\p\R)$. Firstly we  show that $\dim(\R/\mathfrak{P})=\dim(R/\p).$  In fact, suppose that $\dim(\R/\mathfrak{P})<\dim(R/\p).$   Set $k=\dim(\R/\mathfrak{P})$. Then we have by \cite[11.3.3]{BS} that  $\mathfrak{P}\in\Att_{\R}(H^k_{\m\R}(\R/\p\R))=\Att_{\R}(H^k_{\m}(R/\p)).$  Because $\frak P\in \Ass (\R/\p\R)$, we have  $\p=\frak P\cap R\in\Att_R(H^k_{\m}(R/\p))$ by Lemma \ref{L:2a}.  Therefore by the hypothesis (ii) we have  $\p R_{\p}\in\Att_{R_{\p}}\big(H^{k-\dim (R/\p)}_{\p R_{\p}}(R_{\p}/\p R_{\p})\big)$. However, as $\dim (R/\p)>k,$ we have  $\Att_{R_{\p}}\big(H^{k-\dim (R/\p)}_{\p R_{\p}}(R_{\p}/\p R_{\p})\big)=\emptyset$, this is a contradiction. So, $\dim(\R/\mathfrak{P})=\dim(R/\p).$ 

Next, we have  $\dim (\R_{\frak P}/\p \R_{\frak P})=0$  by the above fact. As  $\p\in\Att_R(H^i_{\m}(M))$, we get  by the hypothesis (ii) that $\p R_{\p}\in\Att_{R_{\p}}\big(H^{i-\dim (R/\p)}_{\p R_{\p}}(M_{\p})\big)$.  Note that the natural map $R_{\p}\rightarrow \R_{\frak P}$ is faithfully flat    and 
$$ H^{i-\dim (R/\p)}_{\p R_{\p}}(M_{\p})\otimes \R_{\frak P}\cong H^{i-\dim (\R/\frak P)}_{\frak P\R_{\frak P}}(\widehat M_{\frak P}).$$  Therefore $\frak P \R_{\frak P}\in \Att_{\R_{\frak P}}\big(H^{i-\dim (\R/\frak P)}_{\frak P \R_{\frak P}}(\widehat M_{\frak P})\big)$  by Lemma \ref{L:3a}.  Hence $\frak P\in \Att_{\R}(H^i_{\m\R}(\widehat M))$  by the weak general shifted localization principle, and hence $\frak P\in\Att_{\R}(H^i_{\m}(M))$. Thus,
$$\displaystyle \Att_{\R}(H^i_{\m}(M))\supseteq \bigcup_{\p\in\Att_R(H^i_{\m}(M))}\Ass_{\R}(\R/\p\R).$$

Now we prove the converse inclusion. For each $i\in\{0,\ldots , d-1\}$ with $H^i_{\m}(M)\neq 0,$ there exists by Lemma \ref{L:1a} a prime ideal  $\p\in\Att_R(H^i_{\m}(M))$ such that $\dim (R/\p)=\dim (R/\frak a_i(M)).$ Then $\p R_{\p}\in \Att_{R_{\p}}\big(H^{i-\dim (R/\p)}_{\p R_{\p}}(M_{\p})\big)$ by the assumption (ii). Hence  $H^{i-\dim (R/\p)}_{\p R_{\p}}(M_{\p})\neq 0$ by Lemma \ref{L:1a}(i). So, we have $i\geq \dim (R/\p).$  Hence $\dim (R/\frak a_i(M))\leqslant i$. It follows that $\dim (R/\frak a(M))<d.$ Therefore there exists  an element $x_1\in\frak a(M)^3$ which is a parameter element of $M$. By similar reasons, there exists a system of parameters $(x_1,\ldots , x_d)$ of $M$ such that $x_k\in\frak a(M/(x_1,\ldots ,x_{k-1})M)^3$ for all $k=1,\ldots , d.$ So, by the above claim we have
$$\displaystyle \Att_{\R}(H^i_{\m}(M))\subseteq \bigcup_{\p\in\Att_R(H^i_{\m}(M))}\Ass_{\R}(\R/\p\R).$$

  (iii) $\Rightarrow$ (i).  Let $\p\in\Spec(R)$. Set $t=\dim(R/\p)$ and $\frak a_i(R/\p)=\Ann_R(H^i_{\m}(R/\p))$ for $i=0,1,\ldots, t-1$. As usual we set $\frak a(R/\p) =\frak a_0(R/\p)\ldots \frak a_{t-1}(R/\p).$ Then there exists by Lemma \ref{L:1a}(ii) a prime ideal $\q\in\Att_R(H^i_{\m}(R/\p))$ such that $\dim (R/\q)=\dim (R/\frak a_i(R/\p)).$ Let $\frak Q\in\Ass (\R/\q\R)$ such that $\dim (\R/\frak Q)=\dim (R/\q).$ Then we have by the hypothesis (iii) that $\frak Q\in\Att_{\R}(H^i_{\m}(R/\p)).$  Hence $\frak Q\in\Att_{\R}(H^i_{\m\R}(\R/\p\R)).$ Therefore we get by \cite[Proposition 3.8]{Sh} that $\dim (\R/\frak Q)\leqslant i$, see also \cite{Sch}.  So, $\dim(R/\frak a_i(R/\p))\leqslant i$ for every integer $i=0,\ldots ,t-1$. Hence $\dim(R/\frak a(R/\p))< t$ and hence $\frak a(R/\p)\not\subseteq\p.$ So,  there exists $x\in\frak a(R/\p)\setminus \p.$ Hence $x$ is a parameter element of $R/\p$ and $xH^i_{\m}(R/\p)=0$ for all $i<t.$ It means that $R/\p$ has a uniform local cohomological annihilator, cf. \cite{HH}.  Thus, we have by \cite[ Corollary 4.3]{DJ} that $R$ is universally catenary and all formal fibers of $R$ are Cohen-Macaulay.
 \end{proof}

\end{document}